\documentclass[a4paper,11pt]{article}
\usepackage[T1]{fontenc}
\usepackage[utf8]{inputenc}
\usepackage[english]{babel}
\usepackage{lmodern,amsmath,amssymb,amsthm}
\usepackage[margin=28mm]{geometry}
\usepackage{microtype,xurl}
\usepackage[hidelinks]{hyperref}
\usepackage{enumitem}
\hypersetup{pdftitle={Characterization of Fr\'echet Spaces and Application to Hausdorff MNC},pdfauthor={Henning Wunderlich},pdfsubject={Hausdorff measures of noncompactness in Frechet spaces}}
\newcommand{\K}{\mathbb K}
\newcommand{\R}{\mathbb R}
\newcommand{\N}{\mathbb N}
\newcommand{\co}{\operatorname{co}}
\newcommand{\coc}{\overline{\operatorname{co}}}
\newcommand{\Fix}{\operatorname{Fix}}
\newcommand{\eps}{\varepsilon}
\newcommand{\Frechet}{Fr\'echet}
\newtheorem{theorem}{Theorem}[section]
\newtheorem{proposition}[theorem]{Proposition}
\newtheorem{lemma}[theorem]{Lemma}
\newtheorem{corollary}[theorem]{Corollary}
\theoremstyle{definition}
\newtheorem{example}[theorem]{Example}
\theoremstyle{remark}
\newtheorem{remark}[theorem]{Remark}
\numberwithin{equation}{section}
\setlist[enumerate]{label=\textup{(\roman*)},leftmargin=2em,itemsep=2pt,topsep=4pt}
\title{Characterization of Fr\'echet Spaces\\
and Application to Hausdorff MNC}
\author{Henning Wunderlich\thanks{Offenbach, Germany.
E-mail: \href{mailto:HenningWunderlich@t-online.de}{\texttt{HenningWunderlich@t-online.de}}}}
\date{\today}
\begin{document}
\maketitle

\begin{abstract}
We study the Hausdorff measure of noncompactness associated with compatible
metrics on \Frechet{} spaces, with the aim of extending its Banach-space
properties. After recalling the characterization by complete
translation-invariant metrics with convex balls, we calculate the scalar
Hausdorff measure for a standard weighted fractional metric. For bounded
nonprecompact sets its value is the unique root of an equation involving
a finite seminorm gauge determined by the covering radius. This yields an
actual Hausdorff-condensing criterion for compact perturbations of maps
having separate contraction constants at the seminorm levels. On countable
products of Banach spaces we obtain an exact coordinate formula and a sharp
scalar contraction bound on bounded domains, without requiring a uniform
gap between the coordinate constants and one for strict condensation.
For closed subspaces, an exact finite-projection formula retains the
compatibility constraints on covering centers. Explicit examples show why
separately computed seminorm measures cannot generally be substituted into
the metric formula. Applications to spaces of real sequence-valued continuous
functions give an explicit oscillation formula and exhibit the compatibility
obstruction on an unbounded interval. Classical fixed-point consequences
are recorded briefly.
\end{abstract}


\section{Introduction}

The Hausdorff measure of noncompactness in a Banach space is invariant
under closed convex hulls, is homogeneous and subadditive, and is a basic
tool for condensing operators. In a \Frechet{} space, the topology is
usually described by several seminorms. A compatible metric supplies a
single finite-net measure, but its numerical values and its condensation
properties depend on that metric. Our purpose is to establish precise
versions of the Banach-space statements in this setting.

The underlying locally convex theory is classical. Sadovskii developed
seminorm- and neighborhood-indexed measures and their fixed-point
applications~\cite[\S1.2]{Sadovskii1972}; see also
\cite{TarafdarVyborny1975,Kaniok1990}. Scalar metric measures occur, for
example, in Duong Minh Duc's work on \Frechet{} calculus, where a
diameter-based measure is used~\cite[Example~1.2]{Duc1981}. Metrics with
absolutely convex balls were studied by
Zarnadze~\cite{Zarnadze1987}. The particular fractional metric used below,
in sequence-valued continuous-function spaces, already appears in
Olszowy~\cite[\S\S3--4, pp.~5158--5160]{Olszowy2009}.
That paper also constructs homogeneous scalar MNCs by weighted suprema
of coordinate oscillation measures. These differ from the finite-net
measure of the bounded metric. Related constructions on spaces of
regulated functions are studied in~\cite{DudekOlszowy2022Regulated}.
The metric's reduction to finitely
many seminorm inequalities at a fixed radius is explicit in Dudek and
Olszowy~\cite[equations~(6), (14)]{DudekOlszowy2015}.
The actual scalar Hausdorff MNC is used, for example,
in~\cite[\S2]{KlyusEftekharinasab2019}; its use in a \Frechet{} space
is therefore not itself the contribution considered here.

Families of measures with independent contraction constants $k_n<1$
also support established \Frechet{} fixed-point theorems;
see~\cite{Olszowy2012,Dudek2017,OlszowyDudek2018}.
Recent family-based developments include Meir--Keeler conditions
in~\cite{SoltanpourMajaniSholeHaghighi2025}.
The difficulty of replacing such a family by a single unweighted
supremum is discussed in~\cite[Remark~3.2]{OlszowyDudek2018}.
Countable-product scalar MNCs obtained from suprema of coordinate measures
are also considered by Allahyari, Arab and Shole Haghighi
\cite[Theorem~2]{AllahyariArabSholeHaghighi2020}, who attribute that
construction to~\cite{ArabAllahyariSholeHaghighi2014}.
Related locally convex results for sums of operators are developed
in~\cite{BanasBenAmar2013}. The correction of particular MNC formulas
in~\cite{DudekOlszowy2022Remarks} illustrates the need to distinguish a
specified metric Hausdorff measure from other useful set functions.

The results developed here concern that specified finite-net measure.
For
\[
 d(x,y)=\sup_n a_n\frac{p_n(x-y)}{1+p_n(x-y)},\qquad a_n>0,\quad a_n\to0,
\]
we express its positive value on a bounded set as the unique solution of
$\chi_{q_r}(A)=1$, where $q_r$ is a finite maximum of weighted seminorms.
We then use this equation to obtain scalar Hausdorff condensation from
finite-gauge estimates, including compact perturbations of separate
seminorm contractions. For full countable Banach products, independent
coordinate centers give a simpler exact formula and a sharp contraction
coefficient. A finite-projection theorem explains the modification
required in closed subspaces. The examples demonstrate that ignoring
center compatibility can change the answer even for increasing
equivalent norms. Section~\ref{sec:functions} makes the comparison with
Olszowy's function spaces explicit: it calculates the actual metric
Hausdorff MNC on a bounded interval and shows why separate local
seminorm measures cannot in general be substituted into the metric
formula on an unbounded interval.

Thus the metric construction and the general fixed-point principles are
background; the emphasis is on exact metric MNC formulas and their
consequences. Relative Hausdorff measures and the dependence on the
permitted center set have an established theory
\cite{EspinolaWisnickiWosko2004}, to which the subspace discussion is
explicitly related.

\section{Characterization and Hausdorff MNC}

All spaces are Hausdorff and over $\K=\R$ or $\mathbb C$.
Unless stated otherwise, $E$ denotes a \Frechet{} space.
Boundedness means topological vector space boundedness, unless qualified
otherwise. In a space defined by seminorms $(p_n)$, this means
$\sup_{x\in A}p_n(x)<\infty$ for every $n$. Precompactness means total
boundedness in the canonical uniformity; in a \Frechet{} space it is
equivalent to relative compactness. We use the usual locally convex
terminology of~\cite{Schaefer:TVS}.

For a metric space $X$ and $A\subset X$, write
\begin{equation}\label{eq:alpha}
 \alpha_{X,d}(A)=\inf\{r>0:A\text{ has a finite open }r\text{-net with centers in }X\}.
\end{equation}
For a continuous seminorm $p$ on $E$, define $\chi_p(A)$ by the same
formula with $p(x-y)$ in place of $d(x,y)$ and centers in $E$.
We set both measures of the empty set equal to zero. The value $+\infty$
is allowed for $\chi_p$. It is the norm Hausdorff measure in the
completion of $E/\ker p$: density permits finite centers in the completion
to be approximated in the quotient without changing the infimum.

Put
\begin{equation}\label{eq:h}
 h(t)=\frac{t}{1+t}\quad(0\le t<\infty),\qquad h(\infty)=1.
\end{equation}

\begin{theorem}[Metric characterization]\label{thm:metric}
A complete metrizable topological vector space is \Frechet{} if and only
if it admits a compatible complete translation-invariant metric with
convex balls. If $(p_n)$ is any countable defining separating family of
continuous seminorms and $a_n>0$ tends to zero, then
\begin{equation}\label{eq:metric}
 d(x,y)=\sup_n a_n h(p_n(x-y))
\end{equation}
is such a metric, with absolutely convex balls about zero.
Writing $a_* =\max_n a_n$, define, for $r>0$,
\begin{equation}\label{eq:gauge}
 I_r=\{n:a_n>r\},\qquad
 q_r(x)=\max_{n\in I_r}\frac{a_n-r}{r}p_n(x),
\end{equation}
with an empty maximum equal to zero. Then
\begin{equation}\label{eq:balls}
 B_d(0,r)=\{x:q_r(x)<1\}.
\end{equation}
\end{theorem}
\begin{proof}
The increasing function $h$ is subadditive, since
$h(s+t)=s/(1+s+t)+t/(1+s+t)\le h(s)+h(t)$.
The supremum in \eqref{eq:metric} is finite and satisfies the triangle
inequality; the separating property makes it a metric. The set $I_r$
is finite. For $n\in I_r$, the inequality $a_nh(p_n(x))<r$ is equivalent
to $(a_n-r)p_n(x)/r<1$. Indices with $a_n=r$ impose no constraint,
because $h$ is strictly below one. There are only finitely many indices
with $a_n\ge r$, and the remaining weights have supremum below $r$.
This proves \eqref{eq:balls}, including $r\ge a_*$.

These finite seminorm balls are open and absolutely convex. Convergence
to zero in $d$ is equivalent to convergence in every $p_n$: first control
finitely many indices and then the uniformly small remaining weights.
The same argument applies to Cauchy sequences, proving compatibility
and completeness. Conversely, convex metric balls give a convex
neighborhood base, and hence local convexity.
\end{proof}

\begin{remark}\label{rem:scaling}
The stronger inequality $d(tx,ty)\le t d(x,y)$ for $0\le t\le1$ is
unsuitable for a general \Frechet{} characterization. For a
translation-invariant metric, put $v(x)=d(x,0)$. The triangle inequality
gives
$v(x)\le v(tx)+v((1-t)x)\le v(x)$, forcing $v(tx)=tv(x)$.
Symmetry and reciprocal scaling make $v$ a real norm. In the complex
case $\sup_{|\zeta|=1}v(\zeta x)$ is an equivalent complex norm.
Thus compatibility forces normability. This corrects the stronger
requirement in the original version of this note.
\end{remark}

We recall the finite-net properties needed below. A compatible
translation-invariant metric gives monotonicity, translation and closure
invariance, the finite-union maximum property, and
\begin{equation}\label{eq:sum}
 \alpha_d(A+B)\le\alpha_d(A)+\alpha_d(B).
\end{equation}
It vanishes exactly on precompact sets. For a complete metric, a
decreasing sequence of nonempty closed sets whose measures tend to zero
has a nonempty compact intersection. The seminorm measures likewise have
monotonicity, translation and closure invariance, and the finite-union
maximum property, together with
\begin{equation}\label{eq:semiprops}
 \chi_p(\lambda A)=|\lambda|\chi_p(A),\qquad
 \chi_p(A+B)\le\chi_p(A)+\chi_p(B)
\end{equation}
on bounded sets. Joint vanishing of the $\chi_{p_n}$ is equivalent to
precompactness: finite covers for finitely many seminorms can be
intersected, and one point of $A$ chosen in each nonempty cell gives
a simultaneous net with twice the original radii. These are standard
properties; see~\cite{Sadovskii1972,Akhmerov1992MeasuresON}.

\begin{theorem}[Convex-hull invariance]\label{thm:hull}
For a compatible translation-invariant metric with convex balls on a
locally convex space,
\begin{equation}\label{eq:hull}
 \alpha_d(A)=\alpha_d(\co A)=\alpha_d(\coc A)
 \qquad(A\subset E).
\end{equation}
\end{theorem}
\begin{proof}
For a finite $r$-net $F$ of $A$, convexity of the ball implies
$\co A\subset\co F+B_d(0,r)$. The finite convex hull $\co F$ is compact,
as the continuous image of a finite-dimensional simplex. A finite
$\eps$-net for it is an $(r+\eps)$-net for $\co A$.
Taking infima, then using monotonicity and closure invariance, proves
the assertion. Empty sets and infinite measures cause no exception.
\end{proof}

Every \Frechet{} space therefore admits a metric whose Hausdorff measure
retains the Banach-space convex-hull property. The scalar measures here
have the monotonicity, convex-hull,
zero-set, union, and decreasing-intersection properties of an MNC.
Homogeneity and the convex-combination inequality are retained for the
seminorm measures, and are not imposed on the bounded scalar metric
measure. A Montel \Frechet{} space has zero MNC on every bounded set;
the examples below use non-Montel spaces.

\section{An exact finite-gauge formula}

Throughout this section $d$ and $q_r$ are given by
\eqref{eq:metric}--\eqref{eq:gauge}. A finite maximum of seminorms must
be kept intact when its Hausdorff measure is computed. It cannot in
general be replaced by the maximum of separately computed measures.

\begin{lemma}[Continuity under finite changes of coefficients]\label{lem:continuity}
Let $A$ be bounded, and let $p_1,\ldots,p_N$ be continuous seminorms.
The function
\[
 (c_1,\ldots,c_N)\longmapsto
 \chi_{\max_{1\le j\le N}c_jp_j}(A)
\]
is continuous on $[0,\infty)^N$.
\end{lemma}
\begin{proof}
Fix coefficients $c^0$ and write $q^0=\max_jc_j^0p_j$.
For $c$ sufficiently close to $c^0$, all positive coefficients satisfy
$c_j\ge(1-\eps)c_j^0$. Thus $\max_jc_jp_j\ge(1-\eps)q^0$, giving the
required lower limit. For the upper limit, first choose a finite
$q^0$-net $F$ with radius less than $\chi_{q^0}(A)+\eps$.
Boundedness makes $M_j=\sup_{x\in A,f\in F}p_j(x-f)$ finite.
For the same assigned centers the new seminorm error is at most the old
error plus $\max_j|c_j-c_j^0|M_j$. Let $c\to c^0$, then $\eps\to0$.
\end{proof}

\begin{theorem}[Metric Hausdorff MNC as a level equation]\label{thm:root}
For every $A\subset E$,
\begin{equation}\label{eq:threshold}
 \alpha_d(A)=\inf\{r>0:\chi_{q_r}(A)<1\}.
\end{equation}
If $A$ is topologically bounded and nonprecompact, then $\alpha_d(A)$
is the unique $r\in(0,a_*)$ satisfying
\begin{equation}\label{eq:root}
 \chi_{q_r}(A)=1.
\end{equation}
In particular, for bounded $A$ and $r>0$,
\begin{equation}\label{eq:strict-level}
 \alpha_d(A)<r\quad\Longleftrightarrow\quad\chi_{q_r}(A)<1.
\end{equation}
\end{theorem}
\begin{proof}
If $\chi_{q_r}(A)<1$, a finite seminorm net of radius below one is a
$d$-net of radius $r$, by \eqref{eq:balls}. Conversely, if
$r>\alpha_d(A)$, choose $\alpha_d(A)<s<r$ and a finite $d$-net of
radius $s$. The elementary inequality
\begin{equation}\label{eq:gauge-order}
 q_r\le\frac{s}{r}q_s\qquad(0<s<r)
\end{equation}
shows that the same centers give $\chi_{q_r}(A)\le s/r<1$.
Taking infima proves \eqref{eq:threshold}.

For bounded $A$, put $g_A(r)=\chi_{q_r}(A)$. Near any positive radius,
only finitely many weights can be active, and their coefficients
$(a_n/r-1)_+$ are continuous. Lemma~\ref{lem:continuity} proves
continuity of $g_A$, including at $r=a_n$. Moreover,
$g_A(r)\le(s/r)g_A(s)$ for $s<r$, so it decreases strictly wherever it
is positive. Nonprecompactness gives $\chi_{p_j}(A)>0$ for some $j$,
and therefore
\[
 g_A(r)\ge(a_j/r-1)\chi_{p_j}(A)\longrightarrow\infty
 \quad(r\downarrow0).
\]
Since $g_A(a_*)=0$, it crosses one exactly once.
Equation~\eqref{eq:threshold} identifies that crossing with $\alpha_d(A)$.
For precompact $A$ every seminorm measure is zero, proving the remaining
case of \eqref{eq:strict-level}.
\end{proof}

For a single norm $p$, equation~\eqref{eq:root} reduces to
$\alpha_{a h(p)}(A)=a h(\chi_p(A))$. In the \Frechet{} setting,
the finite gauge $q_r$ takes the place of that one norm.

\begin{remark}[Comparison with a capped metric]\label{rem:capped}
If $(p_n)$ is increasing and $(a_n)$ is nonincreasing, consider the alternative
metric
\[
 d_*(x,y)=\sup_n\min\{a_n,p_n(x-y)\}.
\]
Its Hausdorff measure is given by
\begin{equation}\label{eq:capped}
 \alpha_{d_*}(A)=\sup_n\min\{a_n,\chi_{p_n}(A)\}.
\end{equation}
Indeed, for $0<r\le a_1$ its ball is
$\{p_{N(r)}<r\}$, where $N(r)=\max\{n:a_n\ge r\}$.
If $\alpha_{d_*}(A)<a_n$, finite metric nets of radii decreasing to
$\alpha_{d_*}(A)$ give $\chi_{p_n}(A)\le\alpha_{d_*}(A)$.
Conversely, if $r$ exceeds the right side of \eqref{eq:capped}, then
$\chi_{p_{N(r)}}(A)<r$ yields a finite $d_*$-net of radius $r$;
for $r>a_1$ one center suffices. This proves the formula.
The fractional metric does not permit the same componentwise
substitution, as the next example shows.
\end{remark}

\begin{example}[A finite maximum cannot be separated]\label{ex:two-norms}
Let $E=\R\oplus c_0$ and define increasing equivalent norms
\[
 p_1(s,x)=\max\{|s|,\sup_j|x_j-s|\},\qquad
 p_2(s,x)=|s|+\|x\|_\infty.
\]
For $A=\{(0,e_j):j\ge1\}$,
\begin{equation}\label{eq:two-measures}
 \chi_{p_1}(A)=\tfrac12,\qquad \chi_{p_2}(A)=1,
 \qquad\chi_{\max(2p_1,p_2)}(A)=\tfrac43.
\end{equation}
To verify the last identity, a fixed center $(s,x)$ gives limiting
lower bounds $\max\{|s|,|1+s|\}$ and $1+|s|$ for the respective errors,
because $x_j\to0$. For every $\eps>0$, these bounds minus $\eps$ hold
eventually, simultaneously for all centers of a finite net.
Centers of the form $(-t,0)$ attain them, so the required
infimum is
\[
 \min_{t\ge0}\max\{2\max(t,|1-t|),1+t\}=\tfrac43,
\]
attained at $t=1/3$. The same argument gives the first two identities.

Take $a_1=3,a_2=2$ and, for $n\ge3$, repeat $p_n=p_2$ with positive
weights below $2$ tending to zero. Then
$d=\max\{3h(p_1),2h(p_2)\}$ on differences, and
\begin{equation}\label{eq:two-alpha}
 \alpha_d(A)=\frac{15-\sqrt{33}}8>1
 =\sup_n a_nh(\chi_{p_n}(A)).
\end{equation}
For completeness, if $1<r<3/2$, put $u=(3-r)/r$ and $v=(2-r)/r$.
The same minimization gives
\[
 \chi_{q_r}(A)=\frac{2uv}{u+v}
 =\frac{2(3-r)(2-r)}{r(5-2r)}.
\]
Its level-one equation is $4r^2-15r+12=0$, whose smaller root lies
in this interval. Theorem~\ref{thm:root} proves \eqref{eq:two-alpha}.
Thus increasingness of the defining seminorms does not eliminate the
need for common covering centers.
\end{example}

\section{Scalar condensation from finite-gauge estimates}

Here and below, $T$ is called $\alpha_d$-condensing on a bounded set $C$
if $\alpha_d(TA)<\alpha_d(A)$ for every $A\subset C$ with
$\alpha_d(A)>0$. This definition concerns the actual measure
\eqref{eq:alpha} for the specified ambient space.

\begin{theorem}[Finite-gauge criterion]\label{thm:gauge-condensing}
Let $C\subset E$ be bounded, and let $T:C\to E$ have bounded image.
Suppose that for every $0<r<a_*$ there is $0\le k(r)<1$ such that
\begin{equation}\label{eq:gauge-contract}
 \chi_{q_r}(TA)\le k(r)\chi_{q_r}(A)\qquad(A\subset C).
\end{equation}
Then $T$ is $\alpha_d$-condensing. No uniform bound
$\sup_r k(r)<1$ is required.
\end{theorem}
\begin{proof}
For nonprecompact $A\subset C$, set $r=\alpha_d(A)$.
Theorem~\ref{thm:root} gives $\chi_{q_r}(A)=1$. Thus
$\chi_{q_r}(TA)<1$, and \eqref{eq:strict-level} gives
$\alpha_d(TA)<r$.
\end{proof}

\begin{corollary}[Compact perturbations of seminorm contractions]\label{cor:compact}
Let $S:E\to E$ satisfy
\begin{equation}\label{eq:semi-lip}
 p_n(Sx-Sy)\le k_n p_n(x-y)
 \qquad(x,y\in E),\qquad 0\le k_n<1.
\end{equation}
Let $C\subset E$ be bounded and let $K:C\to E$ have relatively compact
range. Then $T=S|_C+K$ is $\alpha_d$-condensing on $C$.
\end{corollary}
\begin{proof}
For each $r<a_*$, the nonempty index set $I_r$ is finite, and
$k(r)=\max_{n\in I_r}k_n<1$ gives
$q_r(Sx-Sy)\le k(r)q_r(x-y)$.
Mapping a finite seminorm net through the globally defined $S$ yields
$\chi_{q_r}(SA)\le k(r)\chi_{q_r}(A)$.
Since $TA\subset SA+K(A)$ and $\chi_{q_r}(K(A))=0$, subadditivity
proves \eqref{eq:gauge-contract}. Finally, \eqref{eq:semi-lip} makes
$S(C)$ bounded, so $T(C)$ is bounded and the theorem applies.
\end{proof}

The global domain of $S$ in Corollary~\ref{cor:compact} is substantive:
the finite-net centers need not lie in $C$. A map defined only on $C$
requires an appropriate extension or the direct MNC hypothesis
\eqref{eq:gauge-contract}. Also, separate assumptions on
$\chi_{p_n}(TA)$ cannot simply be inserted into
\eqref{eq:gauge-contract}; Example~\ref{ex:two-norms} disproves the
measure identity that would justify that inference. It does not
disprove the existence of some other adapted metric under a more
general family-condensing hypothesis.

\begin{corollary}[Classical fixed-point consequence]\label{cor:fixed}
If $C$ is also nonempty, closed and convex, $T:C\to C$ is continuous,
and the assumptions of Theorem~\ref{thm:gauge-condensing} or
Corollary~\ref{cor:compact} hold, then $\Fix(T)$ is nonempty and compact.
\end{corollary}
\begin{proof}
This is the usual condensing-map argument
\cite{Sadovskii1972,MR1987179}. Briefly, fix $x_0\in C$ and intersect
all closed convex $T$-invariant subsets containing $x_0$. The intersection
$M$ is nonempty and satisfies $M=\coc(\{x_0\}\cup T(M))$.
Theorem~\ref{thm:hull} and the union property give
$\alpha_d(M)=\alpha_d(T(M))$, so condensation makes $M$ compact.
Schauder--Tychonoff gives a fixed point. The full fixed-point set $F$
is closed and satisfies $T(F)=F$; condensation forces its measure to
vanish, proving compactness.
\end{proof}

The fixed-point principle and compact-perturbation arguments have the
predecessors cited above, including~\cite{Duc1981,BanasBenAmar2013}.
The additional conclusion here is condensation for the scalar Hausdorff
measure of \eqref{eq:metric}, obtained from its exact level equation.

\section{Countable products and compatible centers}

Olszowy's weighted oscillation measures on sequence-valued continuous
functions~\cite[Theorems~3.2, 4.2]{Olszowy2009} and the scalar product
construction in~\cite[Theorem~2]{AllahyariArabSholeHaghighi2020}
are predecessors for coordinate aggregation. Related Banach-space
Hausdorff estimates appear in~\cite[Lemmas~3.6--3.7]{XiaoWangLiu2017}.
The equality below concerns the actual finite-net measure of the
specified metric; the subsequent subspace result retains the restrictions
on covering centers.

Let $P=\prod_{j\ge1}E_j$, with each $E_j$ Banach, and let
$w_j>0$ tend to zero. Equip $P$ with
\begin{equation}\label{eq:product-metric}
 d_P(x,y)=\sup_j w_jh(\|x_j-y_j\|_j).
\end{equation}
This is a complete compatible metric with absolutely convex balls.
Write $P_j$ for coordinate projection and $\chi_j$ for the norm
Hausdorff MNC with centers in $E_j$. The sets below need not be
Cartesian products.

\begin{theorem}[Exact product formula]\label{thm:product}
For every $A\subset P$,
\begin{equation}\label{eq:product}
 \alpha_P(A)=\sup_jw_jh\bigl(\chi_j(P_jA)\bigr).
\end{equation}
\end{theorem}
\begin{proof}
For a normed space with metric $w h(\|x-y\|)$, changing the covering
radius gives measure $w h(\chi(A))$, also when $\chi(A)=\infty$.
Coordinate projections therefore prove the lower bound.

Let $r$ exceed the right side of \eqref{eq:product}, and put
$J=\{j:w_j\ge r\}$. For each of these finitely many indices choose a
finite coordinate net of radius $s_j>\chi_j(P_jA)$ with
$w_jh(s_j)<r$. All such measures are finite, since otherwise the
right side would be at least $w_j\ge r$. Combine the coordinate centers
in every possible way and put zero elsewhere. These are valid centers
in $P$ and form a finite $r$-net: the other weights have supremum below
$r$. If $J$ is empty, one center suffices.
\end{proof}

\begin{theorem}[Scalar realization and a sharp coefficient]\label{thm:coordinate-condensing}
Let $C\subset P$ be bounded and let $T:C\to P$ satisfy
\begin{equation}\label{eq:coordinate}
 \chi_j(P_jTA)\le k_j\chi_j(P_jA)
 \qquad(A\subset C),\qquad 0\le k_j<1.
\end{equation}
Then $T$ is $\alpha_P$-condensing, without a uniform gap between the
$k_j$ and one. Moreover, putting $R_j=\chi_j(P_jC)$ gives
\begin{equation}\label{eq:coefficient}
 \alpha_P(TA)\le Q_C\alpha_P(A),\qquad
 Q_C=\sup_{j:R_j>0}\frac{k_j(1+R_j)}{1+k_jR_j},
\end{equation}
where an empty supremum is zero. The bound is sharp from these data:
for products of infinite-dimensional Banach balls and the corresponding
diagonal maps, it is the exact coefficient over all subsets of $C$.
\end{theorem}
\begin{proof}
Let $c_j=\chi_j(P_jA)$ and $m=\alpha_P(A)>0$.
Every $c_j$ is finite. Each term $w_jh(k_jc_j)$ is strictly below $m$:
it is zero when $c_j=0$, and otherwise it is strictly below
$w_jh(c_j)\le m$. Choose $N$ with $\sup_{j>N}w_j<m/2$.
The finite head has maximum below $m$, and the remaining terms are
below $m/2$. Theorem~\ref{thm:product} proves condensation.

For $c>0$, the ratio
\[
 \frac{h(kc)}{h(c)}=\frac{k(1+c)}{1+kc}
\]
is nondecreasing in $c$. Apply this with $c=c_j\le R_j$ to obtain
\eqref{eq:coefficient}. For sharpness, take
$C=\prod_jR_j\overline B_{E_j}$ and $(Tx)_j=k_jx_j$, with the nonzero
factors infinite dimensional. A subset varying only in the $j$th ball
has image-to-domain measure ratio exactly
$k_j(1+R_j)/(1+k_jR_j)$, since the norm Hausdorff measure of that ball
is $R_j$. Taking suprema proves optimality.
\end{proof}

This theorem realizes the familiar coordinate-family conditions as
condensation of the actual scalar metric measure. The usual
fixed-point conclusion follows under continuity, invariance, closedness
and convexity as in Corollary~\ref{cor:fixed}. A uniform bound on both
$k_j<1$ and $R_j$ implies $Q_C<1$; strict condensation can still hold
when $Q_C=1$.

\begin{example}[Condensing with scalar coefficient one]\label{ex:diagonal}
Let $H$ be an infinite-dimensional Hilbert space and put
\[
 P=H^{\N},\qquad C=\prod_j\overline B_H,
 \qquad w_j=2^{-j},\qquad (Tx)_j=\frac{j}{j+1}x_j.
\]
This is a nonnormable \Frechet{} space, and $C$ is closed, convex,
bounded and noncompact. Theorem~\ref{thm:coordinate-condensing} gives
strict condensation, while the exact values are
\begin{equation}\label{eq:diagonal}
 \alpha_P(C)=\tfrac14,\qquad \alpha_P(TC)=\tfrac16,
 \qquad Q_C=\sup_j\frac{2j}{2j+1}=1.
\end{equation}
Thus no scalar $Q<1$ works for all subsets of this one bounded domain.
For the capped metric of Remark~\ref{rem:capped}, with
$p_n(x)=\max_{j\le n}\|x_j\|$ and $a_n=2^{-n}$, instead
$\alpha_{d_*}(C)=\alpha_{d_*}(TC)=1/2$.
The same operator is not condensing for that metric MNC.
For every integer $m\ge0$ the smooth metric also gives the exact value
\[
 \alpha_P(T^mC)=\sup_j2^{-j}h\!\left(\left(\frac{j}{j+1}\right)^m\right).
\]
These numerical statements concern the specified metrics.
\end{example}

\subsection{Closed subspaces and finite projections}

For a closed linear subspace $Y\subset P$, let $\alpha_Y$ use the
inherited metric and centers in $Y$. Put
\[
 \pi_N=P_{\{1,\ldots,N\}},\qquad
 Y_N=\overline{\pi_NY},\qquad
 d_N(u,v)=\max_{j\le N}w_jh(\|u_j-v_j\|_j),
\]
and define
$\beta_N(A)=\alpha_{Y_N,d_N}(\pi_NA)$ and
$b_N=\sup_{j>N}w_j$.

\begin{theorem}[Finite-projection formula]\label{thm:finite}
For every $A\subset Y$ and every $N$,
\begin{equation}\label{eq:finite}
 \beta_N(A)\le\alpha_Y(A)\le\max\{\beta_N(A),b_N\}.
\end{equation}
Consequently $\alpha_Y(A)=\sup_N\beta_N(A)$. If $\alpha_Y(A)>0$,
then $\beta_N(A)=\alpha_Y(A)$ whenever $b_N<\alpha_Y(A)$.
\end{theorem}
\begin{proof}
Projection is nonexpansive, proving the lower bound. Let
$r>\max\{\beta_N(A),b_N\}$. A finite $d_N$-net can be chosen with
radius strictly below $r$. Approximate its centers in $Y_N$ by points
of the dense image $\pi_NY$, preserving that strict margin, and lift
them to $Y$. The errors in the remaining coordinates are at most
$b_N<r$. These lifts form a finite $r$-net in $Y$.
Let $r$ decrease to the displayed maximum. Since $b_N\to0$, both
remaining assertions follow. Projection between successive finite
images also shows that $\beta_N(A)$ is nondecreasing.
\end{proof}

Every \Frechet{} space embeds as a closed linear subspace of a countable
product of Banach spaces, using the completions of its seminorm
quotients. Thus Theorem~\ref{thm:finite} also gives a general
finite-coordinate expression for the metric MNC, provided the centers
are constrained to the compatible images $Y_N$.

The same argument works with a prescribed finite number of centers.
For nonempty $A$ and $m\ge1$, put
\[
 e_m(A;Y)=\inf_{y_1,\ldots,y_m\in Y}
          \sup_{x\in A}\min_{1\le i\le m}d_P(x,y_i),
\]
and define $e_{m,N}$ in $Y_N$ in the same way. Then
$e_{m,N}\le e_m(A;Y)\le\max\{e_{m,N},b_N\}$.
In particular, the finite-projection estimate controls individual
covering radii as well as their limit $\alpha_Y(A)=\inf_m e_m(A;Y)$.

\begin{proposition}[Comparison with independent coordinate centers]\label{prop:relative}
For $A\subset Y$, set $c_j=\chi_j(P_jA)$. Then
\begin{equation}\label{eq:relative}
 \sup_j w_jh(c_j)\le\alpha_Y(A)\le\sup_j w_jh(2c_j).
\end{equation}
The upper envelope is sharp in general, even when all individual
coordinate projections are surjective. The resulting coarser factor
two comparison is also sharp.
\end{proposition}
\begin{proof}
Projection gives the lower bound. Given $r$ above the upper envelope,
take finite coordinate nets of radii $s_j>c_j$ with $w_jh(2s_j)<r$
for the finitely many indices satisfying $w_j\ge r$.
Intersect the corresponding covering cells and choose one point of
$A$ in each nonempty cell. Two points in one cell differ by less than
$2s_j$ in each active coordinate, so these representatives are a finite
$r$-net in $Y$. The inequality $h(2t)\le2h(t)$ gives the coarser bound.
The following example establishes sharpness.
\end{proof}

This recentering argument is the classical source of factor-two
comparisons between relative Hausdorff measures; see
\cite[\S5]{EspinolaWisnickiWosko2004}. Its purpose here is to quantify
the error in applying the full-product formula to a constrained
Fr\'echet subspace.

\begin{example}[Surjective coordinates do not suffice]\label{ex:compatible}
Let
\[
 X=\{(u,v)\in\ell_\infty\times\ell_\infty:u-v\in c_0\},
 \qquad A=\{(s_n,-s_n):n\ge1\},
\]
where $s_n$ has its first $n$ coordinates equal to one and all others
zero. The subspace $X$ is closed and both projections are onto.
Each coordinate set has norm Hausdorff measure $1/2$: a constant
sequence $1/2$ or $-1/2$ supplies the upper bound, and pairwise distance
one supplies the lower bound.

In $X$ with the maximum norm, however, the measure is one. For finitely
many centers $(u^{(i)},v^{(i)})$ and $\eps>0$, choose $k$ large enough
that $|u^{(i)}_k-v^{(i)}_k|<2\eps$ for every center. For all $n\ge k$,
\[
 \max\{|1-u^{(i)}_k|,|-1-v^{(i)}_k|\}
 \ge1-\tfrac12|u^{(i)}_k-v^{(i)}_k|>1-\eps.
\]
The origin gives the reverse bound. With the two active weights equal
to one, the smooth metric is $h$ of the maximum norm, so
\[
 \alpha_X(A)=h(1)=\tfrac12,
 \qquad \sup_jh(c_j)=h(1/2)=\tfrac13,
 \qquad \sup_jh(2c_j)=\tfrac12.
\]
For $tA$ the ratio of the intrinsic measure to the coordinate expression
is $h(t)/h(t/2)$, which tends to two as $t\downarrow0$.
One may append countably many unrestricted Banach factors, with weights
tending to zero and zero coordinates on $A$, to obtain the same
phenomenon in a nonnormable \Frechet{} space.
\end{example}

\section{Continuous-function spaces}\label{sec:functions}

We apply the preceding formulas to the real sequence-valued function
spaces studied in~\cite[\S\S3--4]{Olszowy2009}. The spaces and metrics
in this section come from that paper. Its homogeneous weighted
oscillation measures may be infinite; the metric Hausdorff measures
calculated here are bounded and generally nonhomogeneous.

\subsection{An exact formula on a bounded interval}

For $T>0$ and a nonempty bounded $F\subset C([0,T],\R)$, put
\[
 \omega^T(F,\delta)=\sup_{f\in F}
 \sup_{\substack{s,t\in[0,T]\\|s-t|\le\delta}}|f(s)-f(t)|,
 \qquad
 \omega_0^T(F)=\lim_{\delta\downarrow0}\omega^T(F,\delta).
\]
For the empty set these quantities are zero. This is the oscillation
quantity used in~\cite[\S2]{Olszowy2009}.

\begin{proposition}[Oscillation formula for the metric Hausdorff MNC]
\label{prop:oscillation}
Identify $E_T=C([0,T],\R^{\N})$ with
$\prod_i C([0,T],\R)$ and equip it with
\[
 d_T(x,y)=\sup_i2^{-i}h(\|x_i-y_i\|_{\infty,[0,T]}).
\]
For every topologically bounded $A\subset E_T$,
\begin{equation}\label{eq:oscillation}
 \alpha_{d_T}(A)=\sup_i2^{-i}
 \frac{\omega_0^T(\pi_iA)}{2+\omega_0^T(\pi_iA)}.
\end{equation}
\end{proposition}
\begin{proof}
We verify the real scalar identity
$\chi(F)=\omega_0^T(F)/2$, with centers in $C([0,T],\R)$.
A finite $\eps$-net with centers $g_1,\ldots,g_m$ gives
\[
 \omega^T(F,\delta)\le2\eps+
 \max_{1\le j\le m}\omega^T(\{g_j\},\delta).
\]
Letting $\delta\downarrow0$ and then $\eps\downarrow\chi(F)$
proves the lower bound.

For the reverse bound, choose $b>\omega_0^T(F)/2$ and $\delta>0$
with $\omega^T(F,\delta)<2b$. Take a finite relatively open cover
$(U_j)_{j=1}^m$ of $[0,T]$ with diameters below $\delta$, and a
continuous partition of unity $(\varphi_j)_{j=1}^m$ subordinate to it.
For each $f\in F$, let $c_j(f)$ be the midpoint of its supremum and
infimum on $U_j$. Then
\[
 g_f=\sum_{j=1}^m c_j(f)\varphi_j,
 \qquad \|f-g_f\|_\infty<b.
\]
The functions $g_f$ form a bounded subset of the finite-dimensional
span of the $\varphi_j$, and hence have finite nets of arbitrarily
small radius. Thus $\chi(F)\le b$, proving the scalar identity.
Topological boundedness of $A$ makes every $\pi_iA$ norm bounded.
Theorem~\ref{thm:product} now gives \eqref{eq:oscillation}.
\end{proof}

Olszowy's measure on this space has the form
$\sup_i b_i\omega_0^T(\pi_iA)$ for positive weights $b_i$
\cite[p.~5159]{Olszowy2009}. Formula~\eqref{eq:oscillation}
identifies the finite-net measure of $d_T$ instead. It is bounded by
$1/2$ and has a nonlinear dependence on the coordinate oscillations.
The boundedness hypothesis is essential to the oscillation identity:
an unbounded family of constant functions can have zero oscillation
measure and infinite norm Hausdorff MNC.

\subsection{Compatible centers on an unbounded interval}

The local restrictions on different time intervals must arise from
the same continuous function. The following example computes the
effect of this requirement for the metric in
\cite[\S4, p.~5160]{Olszowy2009}.

\begin{example}[Failure of the separate-seminorm formula]
\label{ex:ramps}
Let $E=C(\R_+,\R^{\N})$ and put
\[
 p_n(x)=\max_{i\le n}\sup_{0\le t\le n}|x_i(t)|,
 \qquad d(x,y)=\sup_n2^{-n}h(p_n(x-y)).
\]
For $m\ge1$ set
\[
 f_m(t)=\min\{1,m(t-1)_+\},\qquad
 x^{(m)}=(f_m,0,0,\ldots),\qquad
 A=\{x^{(m)}:m\ge1\},
\]
where $u_+=\max\{u,0\}$. The set $A$ is topologically bounded, and
\begin{equation}\label{eq:ramp-seminorms}
 \chi_{p_1}(A)=0,\qquad
 \chi_{p_n}(A)=\tfrac12\quad(n\ge2).
\end{equation}
Nevertheless,
\begin{equation}\label{eq:ramp-alpha}
 \alpha_d(A)=\frac{3-\sqrt3}{12}>\frac1{12}
 =\sup_n2^{-n}h(\chi_{p_n}(A)).
\end{equation}

Indeed, all $f_m$ vanish on $[0,1]$. For $n\ge2$ the constant first
coordinate $1/2$, with other coordinates zero, gives the upper bound
in \eqref{eq:ramp-seminorms}. For any continuous center $g$, write
$c=g_1(1)$. Since $f_m(1)=0$, $f_m(1+1/m)=1$ and
$g_1(1+1/m)\to c$, the limiting lower bound for the $p_n$-error is
$\max\{|c|,|1-c|\}\ge1/2$. For finitely many centers this bound holds
simultaneously, proving \eqref{eq:ramp-seminorms}.

The same two evaluations give
\[
 \liminf_{m\to\infty}d(x^{(m)},g)
 \ge\max\left\{\tfrac12h(|c|),\tfrac14h(|1-c|)\right\}.
\]
Projection of $c$ onto $[0,1]$ cannot increase either term. On that
interval the first term increases and the second decreases, so the
minimum occurs where they agree, at $c_*=2-\sqrt3$. Their common
value is $(3-\sqrt3)/12$. The lower bound remains valid for the
minimum over any finite set of centers. Conversely, the constant
center $g=(c_*,0,0,\ldots)$ has $p_1$-error $c_*$ and $p_n$-error
$1-c_*$ for every $n\ge2$, and attains this value of $d$ for all
$x^{(m)}$. This proves \eqref{eq:ramp-alpha}.
\end{example}

Example~\ref{ex:ramps} explains the role of the joint gauge in
Theorem~\ref{thm:root} within a natural nonnormable function space.
The separate seminorm infima use centers that need not be compatible;
the joint gauge preserves the common-center requirement.

\section*{Disclosure of AI assistance}
OpenAI ChatGPT was used as an assistance tool for drafting, checks, and locating related literature.  
The author reviewed, revised, and takes responsibility for the mathematical statements, arguments, citations, and final contents of the manuscript.

\bibliographystyle{alpha}
\bibliography{Frechet-MNC}
\end{document}